\documentclass[reqno,11pt]{amsart}

\usepackage{amsmath}
\usepackage{amssymb}
\usepackage{enumerate}
\usepackage{a4wide}
\usepackage{pdfsync}

\newtheorem{theorem}{Theorem}[section]

\newtheorem{lemma}[theorem]{Lemma}

\theoremstyle{remark}
\newtheorem{remark}[theorem]{Remark}

\newcommand{\R}{\mathbb{R}}
\newcommand{\cR}{\mathcal{R}}

\newcommand{\cD}{\mathcal{D}}

\newcommand{\pt}{\partial_{t}}
\newcommand{\del}{\partial}
\newcommand{\pa}{\partial_{\alpha}}

\newcommand{\e}{\varepsilon}
\newcommand{\eps}{\varepsilon}
\renewcommand{\div}{\textrm{div}}
\newcommand{\vh}{\widehat{v}}
\newcommand{\fh}{\widehat{F}}
\newcommand{\yh}{\widehat{y}}
\newcommand{\xh}{\widehat{\Xi}}

\newcommand{\cof}{\mathop{\mathrm{cof}}}

\newcommand{\Real}{\R}

\newcommand{\mat}{\hbox{Mat}}
\newcommand{\dt}{\hbox{det}\,}
\numberwithin{equation}{section}

\begin{document}
\title [relaxation with polyconvex entropy]
{
Stress relaxation models with polyconvex entropy in
Lagrangean and Eulerian coordinates
%
}

\author{Athanasios E. Tzavaras}

\address{Department of Applied Mathematics, University of 
Crete, Heraklion, Greece and Institute for Applied and Computational Mathematics, FORTH\@, Heraklion, Greece}

\email{tzavaras@tem.uoc.gr}

\date{}

\begin{abstract}
The embedding of the equations of polyconvex elastodynamics to an augmented 
symmetric hyperbolic system provides in conjunction with the relative entropy method
 a robust stability framework for approximate solutions \cite{LT06}.
We devise here a model of  stress relaxation motivated by the 
format of the enlargement process which formally approximates 
the equations of polyconvex elastodynamics.  The model is endowed with 
an entropy function which is not convex but rather of polyconvex type.
Using the relative entropy we prove a stability estimate and convergence
of the stress relaxation model to polyconvex elastodynamics in the
smooth regime. As an application, we show that models of pressure relaxation for
real gases in Eulerian coordinates fit into the proposed framework.
\end{abstract}

\dedicatory{To Marshall Slemrod on the occasion of his 70th birthday with friendship and admiration}

\date{}
\maketitle
\baselineskip 14pt

\section{Introduction}
\label{intro}

The mechanical motion of a continuous medium with nonlinear elastic response 
is described by the system of partial differential equations
    \begin{equation}
        \frac{\del^{2} y}{\del t^{2}} = \nabla \cdot T (\nabla {y})
        \label{displacement}
    \end{equation}
where $y : \R^{3} \times \R^+ \to \R^{3}$ describes the motion,
$T$ is the Piola--Kirchhoff stress tensor,  $v = \del_{t} y$ is the velocity
and $F = \nabla y$ the deformation gradient.
Motivated by the requirements imposed on the theory of thermoelasticity
from consistency with the Clausius-Duhem inequality of thermodynamics, one often
imposes the assumption of hyperelasticity, i.e. that $T$ is expressed as a gradient
$T (F)  =  \frac{\del W }{\del F} (F)$ of the stored energy function $W : \text{Mat}^{3 \times 3} \to [0, \infty)$.
The principle of material frame indifference dictates that $W$
remains invariant under rotations
$$
W(O F) = W(F) \quad \text{for all orthogonal matrices $O \in O(3)$.}
$$

Convexity of the stored energy $W$ is too restrictive and even 
incompatible with certain physical requirements:
First, it conflicts with frame indifference in conjunction with the requirement 
that the energy increase without bound as $\dt F\to 0^+ $.
Second, convexity of the energy
together with the axiom of frame indifference impose 
restrictions on the induced Cauchy stresses that rule out certain
naturally occurring states of stress
({\it e.g.} \cite[Sec 8]{CN59}, \cite[Sec 4.8]{Cia93}).
As a result, it has been 
replaced in the theory of elastostatics by various weaker notions 
such as quasi-convexity, rank-1 convexity or polyconvexity, see
\cite{Bal77} or \cite{Bal02} for a recent survey.
Here, we adopt the assumption of polyconvexity 
which postulates that 
$$
W(F) = g (F, \cof F, \dt F) \, ,
$$
where $g$ is a strictly convex function of 
$\Phi(F) = (F , \cof F , \dt F)$,
and encompasses various interesting models ({\it e.g.} \cite{Cia93}).

The system (\ref{displacement}) may be recast as a system of 
conservation laws, for the velocity $v = \del_{t} y$ and the 
deformation gradient $F = \nabla y$, in the form
\begin{equation}
       \begin{aligned}
            \pt F_{i\alpha} &= \pa v_{i}   \\
           \pt v_{i} &= \pa T_{i\alpha}(F) \, ,
       \end{aligned}
       \label{elas}
\end{equation}
$i, \alpha = 1, \ldots, 3$.
The equivalence holds for solutions $(v,F)$ with $F = \nabla y$, 
i.e. subject to the set of differential constraints
\begin{equation}
\label{gconstraint}
\partial_{\beta} F_{i \alpha} - \partial_{\alpha} F_{i \beta} = 0 \, .
\end{equation}
Equation \eqref{gconstraint} is an involution:
if it holds initially it is propagated by \eqref{elas}$_{1}$
to hold for all times. The system \eqref{elas} is endowed with an additional conservation law
\begin{equation}
\del_t ( \frac{1}{2} v^{2} + W(F) ) - \del_\alpha ( v_i T_{i \alpha} (F) ) = 0
\end{equation}
manifesting the conservation of mechanical energy.
When $W$ is convex the "entropy" $E = \frac{1}{2} v^{2} + W(F)$ is a convex function.
Convexity of the entropy is known to provide a stabilizing mechanism for
thermomechanical processes, and entropy inequalities for convex entropies
have been employed in the theory of hyperbolic conservation laws as an 
admissibility criterion for weak solutions \cite{Lax71} and 
provide powerful stability frameworks for approximations of  classical solutions \cite{Daf79}, \cite{DiP79}. 
Such stability is attained via the relative entropy method and applies in 
particular to viscosity or even relaxation approximations of the system \eqref{elas}, \cite{LT06}, \cite[Ch V]{Daf10}.

By contrast, when $W$ is not convex the entropy $E = \frac{1}{2} v^{2} + W(F)$ is also non-convex,
what induces an array of questions regarding the stability of the model within its various approximating theories. 
One should distinguish between models where one tries to model inherently unstable phenomena
(like for example phase transitions) and models where one expects stable response but where the
invariance under rotations imposes degeneracies (like the problem of elasticity). 
Our objective is to contribute to a program \cite{Qin98, DST01, LT06} of understanding such issues 
and to suggest remedies especially as it pertains to the stable approximation
of elastodynamics by stress relaxation theories.

Relaxation approximations encompass many physical models and
have proved useful in designing efficient algorithms for systems of conservation laws 
({\it e.g.} \cite{CP97, Bou05, BKW07}) while convexity of the entropy is known to provide 
a stabilizing effect for general relaxation approximations  ({\it e.g} \cite{CLL94}, \cite{Tza04}).  
A natural relaxation approximation of  \eqref{elas} is given by the stress relaxation
theory
\begin{equation}
       \begin{aligned}
              \pt F_{i\alpha} &= \pa v_{i}   \\
             \pt v_{i} &= \pa S_{i\alpha}      \\
\pt(S_{i\alpha} - f_{i\alpha}(F)) &= -\frac{1}{\e}(S_{i\alpha}
-T_{i\alpha}(F)).
\end{aligned}
           \label{isorelax}
\end{equation}
This model may be visualized within the framework of viscoelasticity
with memory
$$
S = f(F) + \int_{-\infty}^{t} \frac{1}{\e}
     e^{-\frac{1}{\e} (t-\tau)} h(F(\cdot, \tau)) \,
     d\tau \,
$$
with the equilibrium stress $T(F)$ decomposed into an elastic
and viscoelastic contribution, $T(F) = f(F) + h(F)$, where
$f = \frac{\del W_{I}}{\del F}$ and $T = \frac{\del W}{\del F}$,
and a kernel exhibiting a single relaxation time $\eps$.
It belongs to the class of thermomechanical theories with internal variables which 
have been extensively studied in the mechanics literature {\it e.g.} \cite{CG67, FM87, Sul98, Tza99}.

The approximation \eqref{isorelax} is consistent
with the second law of thermodynamics,
provided the potential of the instantaneous elastic response $W_{I}$
dominates the potential of the equilibrium response $W$.
When $W$ is convex the relaxation theory has a convex entropy 
and a relative entropy calculation indicates that \eqref{isorelax} stably approximates \eqref{elas}  \cite{LT06}.
On the other hand, for polyconvex  $W$, the consistency with thermodynamics is
still attained but the entropy of the relaxation system loses convexity and the
stability of the approximating system is questionable.
Convexity of the entropy is a dictum of stability for relaxation 
approximations; at the same time it is not a consequence of 
thermodynamical consistency of relaxation theories with the Clausius-Duhem
inequality \cite{Sul98, LT06}.
As convexity is largely incompatible with material frame indifference, 
the effect of adopting weaker notions of convexity 
on the stability of thermomechanical processes needs to be further understood.

Our objective  is to propose a stable relaxation approximation scheme for the equations of polyconvex
elasticity. We will be guided by the embedding of polyconvex elasticity to an augmented
strictly hyperbolic system:
Due to nonlinear transport identities of the null-Lagrangians, the system \eqref{elas} with
polyconvex stored energy can be embedded into an augmented symmetric hyperbolic system 
\begin{equation}
\begin{aligned}
\label{enelas}
\del_{t} v_{i} &=
\del_{\alpha} \biggl(
\frac{\partial g}{\partial\Xi^A}(\Xi)\frac{\partial\Phi^A}{\partial
F_{i\alpha}}(F)\biggr)
\\
\del_{t} \Xi^{A} &= \del_{\alpha} \biggl(\frac{\partial\Phi^A}{\partial
F_{i\alpha}}(F)v_i\biggr)
\end{aligned}
\end{equation}
and be visualized as constrained evolution thereof (see \cite{Qin98, DST01} and section \ref{secvisco}
for an outline).  The augmented system admits the convex entropy 
$\eta = \frac{1}{2}|v|^{2} + g (\Xi)$ and is symmetrizable.
The idea of symmetrizable extensions of \eqref{elas} has important implications on the equations of  polyconvex elasticity,
providing stability frameworks between entropy weak and smooth solutions  \cite{LT06},  \cite[Ch V]{Daf10}
or even between entropic measure-valued and smooth solutions \cite{DST12}.
The idea of enlarging the number of variables and  extending to  symmetrizable hyperbolic systems  
has been fruitful in other contexts like for nonlinear models
of electromagnetism \cite{Bre04,Ser04} or for the isometric embedding problem in geometry \cite{Sle13}.

In the sequel, we consider the stress relaxation system 
\begin{equation}
\label{relpoly0}
\begin{aligned}
\del_{t} v_{i} - 
\pa \left ( T^{A} \frac{\partial \Phi^{A}}{\partial F_{i\alpha}} 
\right ) &= 0
\\
\del_{t} F_{i \alpha} - \pa v_{i} &= 0
\\
\del_{t} \left ( T^{A} -  \frac{\del \sigma_{I}}{\del \Xi^{A} } (\Phi(F) )\right )
&= - \frac{1}{\eps} 
\left ( T^{A} -  \frac{\del \sigma_{E}}{\del \Xi^{A} } (\Phi(F) )\right )
\\
\del_{\beta} F_{i \alpha} - \del_{\alpha} F_{i \beta} &= 0 \, .
\end{aligned}
\end{equation}
The format of \eqref{relpoly0} is motivated 
by an attempt to transfer the geometric structure of the limit to the approximating 
relaxation system.
Note that \eqref{relpoly0} formally approximates as $\eps \to 0$ 
the equations of polyconvex elastodynamics and retains the 
property of embedding into to an augmented relaxation system  
(see \eqref{polyrelax}) with the latter endowed with an entropy dissipation inequality
for a  convex entropy. The  reduced entropy inherited by \eqref{relpoly0} 
is of the form 
$$
\mathcal{E}= \frac{1}{2} |v|^{2}+ \Psi(\Phi(F), \tau)
$$
with $\Psi(\Xi,\tau)$ a convex function,  $\Phi(F)=(F, \cof F, \dt F)$, and thus $\mathcal{E}$
is  of polyconvex type. 
We prove
using a relative entropy computation and the null-Lagrangian structure
that this theory approximates in a stable way smooth solutions of 
\eqref{elas} with polyconvex stored energy. Our analysis indicates that it is possible to
stabilize a relaxation model via a globally defined, polyconvex entropy.

The system \eqref{relpoly0} appears unconventional as it mixes geometric and mechanical properties.
Nevertheless, it contains a very interesting example. When the equations of  isentropic
gas dynamics in Eulerian coordinates are adapted to this model, 
and after performing the proper transformations from 
Eulerian to Lagrangean coordinates, one achieves a model of relaxation of pressures
(see \eqref{gdpressrelax}) with an instantaneous and an equilibrium pressure response.
The latter is endowed with a globally defined, dissipative, convex entropy. Models of pressure
relaxation have been considered before in \cite{CG67, Sul98}. The novelty of the present one is the
existence of a global, convex entropy. This is in a similar spirit (but a different model) as the model of internal energy 
relaxation for gas dynamics pursued in \cite{CP97}. 

The article is organized as follows. In Section \ref{secvisco} we 
present the embedding of \eqref{elas} into the augmented system 
\eqref{enelas} and define the relative entropy. 
In section \ref{secrel} we state the augmented 
relaxation system \eqref{polyrelax}, show that it
is endowed with a convex entropy, and exhibit the inherited relative 
entropy calculation \eqref{relenpoly} for the system \eqref{relpoly0}.
This culminates into the stability and convergence Theorem 
\ref{theorelax} between solutions of the relaxation model \eqref{relpoly0}
and the polyconvex elastodynamics system \eqref{elas}.
As an application of the theory, in section \ref{secgd}, we develop an example of
pressure relaxation that converges to the equations of isentropic gas dynamics in Eulerian 
coordinates and is endowed with a convex entropy function.

The results of sections \ref{secrel} and \ref{secstab} are taken from an earlier unpublished version of this 
manuscript \cite{Tza05}.

\section{The symmetrizable extension of polyconvex elastodynamics}
\label{secvisco}

The system of elastodynamics \eqref{displacement} is
expressed in the form of a system of conservation laws \eqref{elas}, \eqref{gconstraint}.
As already noted, the equivalence of the two formulations holds for functions $F$ that
are gradients, but as the relation 
   $F = \nabla y$ propagates from the initial data, 
relation \eqref{gconstraint} is viewed as a constraint on the initial data
and is usually omitted.  We work under the framework of polyconvex hyperelasticity:
the Piola-Kirchoff stress is derived from a potential
$
T (F) = \frac{\del W(F)}{\del F}
$
and the stored energy
$W :\mat^{3\times 3}\to [0,\infty)$ factorizes as a 
function of the minors of $F$,
\begin{equation}
\label{polyconvex}
\begin{aligned}
W (F)= \big ( g\circ\Phi \big ) (F) \, , 
\quad \text{where} \; 
&\Phi(F)=(F,\cof F,\dt F) \, , 
\end{aligned}
\end{equation}
with $g:\mat^{3\times 3}\times \mat^{3\times 3}\times\Real\to\Real$ convex.
The cofactor matrix $\cof F$ and the determinant $\dt F$ are
\begin{align*}
\begin{split}
(\cof F)_{i\alpha} &=
\frac{1}{2}\epsilon_{ijk}\epsilon_{\alpha\beta\gamma}
F_{j\beta}F_{k\gamma}\, ,\\
\dt F &=\frac{1}{6}\epsilon_{ijk}\epsilon_{\alpha\beta\gamma}
F_{i\alpha}F_{j\beta}F_{k\gamma} = \frac{1}{3}(\cof F)_{i \alpha}F_{i 
\alpha}\, .
\end{split}
\end{align*}

We review a symmetrizable extension of polyconvex elastodynamics
\cite{DST01}, based on certain kinematic identities
on $\det F$ and $\cof F$ from \cite{Qin98}.
The components of $\Phi(F)$ 
are null Lagrangians and satisfy
 the identities
\begin{equation*}
\frac{\partial}{\partial x^\alpha}
\biggl(\frac{\partial\Phi^A}{\partial F_{i\alpha}}
(\nabla y)
\biggr)
\equiv 0
\end{equation*}
for any smooth map $y(x,t)$. Equivalently, this is expressed as
\begin{equation}
\label{nullag2}
\del_{\alpha}
\biggl(\frac{\partial\Phi^A}{\partial F_{i\alpha}} (F)
\biggr)
= 0 \, ,
\quad \forall F \; \text{with} \;
\partial_{\beta} F_{i \alpha} - \partial_{\alpha} F_{i \beta} =0\, .
\end{equation}
The kinematic compatibility equation \eqref{elas}$_{1}$ implies
\begin{equation}
\begin{aligned}
\del_{t} \Phi^{A} (F) -
\del_{\alpha}
\biggl( v_i \frac{\partial\Phi^A}{\partial F_{i\alpha}}(F)\biggr) = 0
\, .
\label{kinnl}
\end{aligned}
\end{equation}
Strictly speaking \eqref{kinnl} do not form what is called in the theory of
conservation laws entropy - entropy flux pairs as they hold only for
$F$ that are gradients, i.e. $\forall F$ with
$\partial_{\beta} F_{i \alpha} - \partial_{\alpha} F_{i \beta} =0$.

This motivates to embed \eqref{elas}, \eqref{kinnl} into the system of conservation laws
\begin{equation}
\begin{aligned}
\label{enlarged}
\del_{t} v_{i} &=
\del_{\alpha} \biggl(
\frac{\partial g}{\partial\Xi^A}(\Xi)\frac{\partial\Phi^A}{\partial
F_{i\alpha}}(F)\biggr)
\\
\del_{t} \Xi^{A} &= \del_{\alpha} \biggl(\frac{\partial\Phi^A}{\partial
F_{i\alpha}}(F)v_i\biggr)\, .
\end{aligned}
\end{equation}
Note that $\Xi = (F,Z,w)$ takes values in
$\mat^{3\times 3}\times \mat^{3\times 3}\times\Real \simeq \R^{19}$
and is treated as a new dependent variable.
The extension has the following properties:
\begin{itemize}
\item[(i)] If $F(\cdot, 0)$ is a gradient then $F(\cdot, t)$ remains a
gradient $\forall t$.

\item[(ii)] If $\Xi (\cdot, 0) = \Phi (F(\cdot, 0))$ with
$F(\cdot, 0) = \nabla y_{0}$, then  
$\Xi (\cdot, t) = \Phi (F(\cdot, t))$ where $F(\cdot, t) = \nabla y (\cdot , t)$.
In other words, the system of elastodynamics can be visualized as
constrained evolution of \eqref{enlarged}.

\item[(iii)] The enlarged system admits a strictly convex entropy
\begin{equation*}
\eta(v,\Xi)=\frac{1}{2}|v|^2 + g(\Xi)
\end{equation*}
and is thus symmetrizable (along solutions that are gradients).

\item[(iv)] The system is endowed with a relative entropy calculation, detailed below.
\end{itemize}

Property (iii) is  based on the null-Lagrangian structure
and $\eta$ is not an entropy in the usual sense
of the theory of conservation laws. Rather, the identity 
        \begin{equation}
        \label{enlencons}
            \pt \left [ \frac{1}{2}|v|^{2} + g(\Xi)\right ] - \pa
            \left[ \sum_{i,A} v_{i}\frac{\partial g(\Xi)}{\partial
       \Xi^{A}}\frac{\partial \Phi^{A}(F)}{\partial
       F_{i\alpha}}\right ] = 0
        \end{equation}
holds for $F$'s that are gradients.

Property (iv)  pertains to the following {\bf relative energy} calculation \cite{LT06}, \cite [Ch V]{Daf10}.
Let $y$ be an entropic weak solution satisfying the weak form of \eqref{elas}, \eqref{gconstraint} and the weak form
of the entropy inequality 
        \begin{equation*}
            \pt \left [ \frac{1}{2}|v|^{2} + g(\Phi(F) )\right ] - \pa
            \left[ \sum_{i,A} v_{i}\frac{\partial g(\Phi(F))}{\partial
       \Xi^{A}}\frac{\partial \Phi^{A}(F)}{\partial
       F_{i\alpha}}\right ] \le 0   \quad \mbox{in $\cD'$}.
        \end{equation*}
Then provided $F = \nabla y$ enjoys sufficient integrability properties, $F$ also satisfies the weak form of \eqref{kinnl}.
As a result $(v, \Xi)$ with $\Xi = \Phi(F)$ is a weak solution of \eqref{enlarged} which is entropic in the sense that
\begin{equation*}
            \pt \left [ \frac{1}{2}|v|^{2} + g(\Xi)\right ] - \pa
            \left[ \sum_{i,A} v_{i}\frac{\partial g(\Xi)}{\partial
       \Xi^{A}}\frac{\partial \Phi^{A}(F)}{\partial
       F_{i\alpha}}\right ] \le 0  \quad \mbox{in $\cD'$}.
        \end{equation*}
Let $\yh$ be a smooth solution of \eqref{displacement}. Then $(\vh, \fh)$ satisfy \eqref{elas}, \eqref{gconstraint}
and the augmented function $(\vh,  \xh)$ with $\xh = \Phi (\fh)$ satisfies the energy conservation \eqref{enlencons}.
Then, the two solutions $(v, \Phi(F))$ and $(\vh, \Phi(\fh))$ can be compared via the relative energy formula
\begin{equation}
       \label{eq:relativeentropyvisc}
\begin{aligned}
     & \del_{t} \Big ( \eta  (v,\Phi(F)  \;  | \;  \vh, \Phi(\fh) ) \Big )   - \nabla \cdot  \Big (  q (v,\Phi(F)  \;  | \;  \vh, \Phi(\fh) ) \Big ) 
           \le Q   \, , 
\end{aligned}
\end{equation}
where 
\begin{equation*}
\begin{aligned}
\eta (v,\Xi \;  | \;  \vh,\xh) &:=
\frac{1}{2} |v-\vh|^{2} + g(\Xi) - g(\xh)
            -\frac{\partial g(\xh)}{\partial \Xi^{A}}(\Xi^{A} - \xh^{A}) \, , 
\\
q^{\alpha} (v,\Xi \;  | \;  \vh,\xh)  &: =
   \left (\frac{\partial g(\Xi)}{\partial \Xi^{A}}
-\frac{\partial g(\xh)}{\partial
\Xi^{A}}\right)(v_{i}-\vh_{i})\frac{\partial \Phi^{A}(F)}{\partial
F_{i\alpha}}\, ,  \quad \mbox{$\alpha = 1,2,3$} \, , 
\end{aligned}
\end{equation*}
and  $Q$ is a quadratic error term of the form
\begin{align}
     &Q   : =  
     \Big [ \frac{\partial^{2}g }{\partial \Xi^{A}\partial \Xi^{B}}  ( \Phi  (\fh))  \Big ]
     \pa ( \Phi^{B} (\fh)) \left (\frac{\partial \Phi^{A}(F)}{\partial
F_{i\alpha}} - \frac{\partial \Phi^{A}(\fh)}{\partial
F_{i\alpha}} \right )(v_{i}-\vh_{i})
\nonumber\\
&\quad\ +
(\pa \vh_{i} ) \left (\frac{\partial g(\Phi (F) )}{\partial \Xi^{A}}
-\frac{\partial g( \Phi (\fh) )}{\partial
\Xi^{A}}\right)\left (\frac{\partial \Phi^{A}(F)}{\partial
F_{i\alpha}} - \frac{\partial \Phi^{A}(\fh)}{\partial
F_{i\alpha}} \right )
           \nonumber\\
            & \quad\ +  ( \pa \vh_{i} ) 
\Big ( \frac{\partial
g(\Xi)}{\partial \Xi^{A}}
- \frac{\partial g(\xh)}{\partial
\Xi^{A}} - \frac{\partial^{2}g(\xh)}{\partial \Xi^{A}\partial
            \Xi^{B}}(\Xi^{B} - \xh^{B}) 
\Big ) 
\Bigg |_{\Xi = \Phi(F), \; \xh = \Phi (\fh) }
\frac{\partial \Phi^{A}(\fh)}{\partial
F_{i\alpha}}\, .
     \label{quadraticQ}
       \end{align}
Details of the lengthy computation can be found in \cite{LT06}
and use in a substantial way the null-Lagrangian identity \eqref{nullag2}.
There is also available an analogous formula for comparing entropic (or dissipative) measure-valued solutions
to smooth solutions of \eqref{elas}, see \cite{DST12}.


\section{A relaxation model  for polyconvex elastodynamics}
\label{secrel}

We next consider the stress relaxation model
\begin{equation}
\label{relpoly}
\begin{aligned}
\del_{t} v_{i} - 
\pa \left ( T^{A} \frac{\partial \Phi^{A}}{\partial F_{i\alpha}} 
\right ) &= 0
\\
\del_{t} F_{i \alpha} - \pa v_{i} &= 0
\\
\del_{t} \left ( T^{A} -  \frac{\del \sigma_{I}}{\del \Xi^{A} } (\Phi(F) )\right )
&= - \frac{1}{\eps} 
\left ( T^{A} -  \frac{\del \sigma_{E}}{\del \Xi^{A} } (\Phi(F) )\right )
\\
\del_{\beta} F_{i \alpha} - \del_{\alpha} F_{i \beta} &= 0 \, , 
\end{aligned}
\end{equation}
and wish to compare it to the equations of elastodynamics 
\begin{equation}
\label{pelas}
\begin{aligned}
\del_{t} v_{i} - 
\pa \left (  \frac{\del \sigma_{E}}{\del \Xi^{A} } (\Phi(F) )
\frac{\partial \Phi^{A}}{\partial F_{i\alpha}} 
\right ) = 0
\\
\del_{t} F_{i \alpha} - \pa v_{i} = 0 \, .
\end{aligned}
\end{equation}
The stress in the model \eqref{pelas} satisfies
$$
S_{\infty} = \frac{\del}{\del F} \sigma_{E} (\Phi (F))
$$
and thus, when $\sigma_{E}$ is convex, the model \eqref{pelas} corresponds
to polyconvex elasticity.

The model \eqref{relpoly} corresponds to a stress relaxation theory
where the stress is decomposed into an instantaneous and a 
viscoelastic part
\begin{equation}
S = T^{A} \frac{\partial \Phi^{A}}{\partial F} = \frac{\del 
( \sigma_{I} \circ \Phi )}{\del F} 
+ \tau^{A} \frac{\partial \Phi^{A}}{\partial F}
\end{equation}
and where the instantaneous elasticity is derived from a polyconvex potential
$\sigma_{I} (\Phi(F))$ while the viscoelastic part is determined by 
internal variables $\tau^{A}$ evolving according to the model
\begin{equation}
\label{intvode}
\del_{t} \tau^{A} = - \frac{1}{\eps} \Big ( \tau^{A} - 
\frac{\del (\sigma_{E} - \sigma_{I})}{\del \Xi^{A}}  (\Phi (F)) \Big ) \, .
\end{equation}

Note that when expressed in terms of the 
motion $y$ the model \eqref{relpoly} takes the form
\begin{equation}
\begin{aligned}
\frac{\del^{2}y}{\del t^{2}} &= \nabla \cdot \Big ( 
\frac{\del 
( \sigma_{I} \circ \Phi )}{\del F} (\nabla y)
+ \tau^{A} \frac{\partial \Phi^{A}}{\partial F}(\nabla y)
\Big )
\\
\frac{\del \tau^{A}}{\del t} &= - \frac{1}{\eps} \Big ( \tau^{A} - 
\frac{\del (\sigma_{E} - \sigma_{I})}{\del \Xi^{A}}  (\Phi (\nabla y)) \Big )
\end{aligned}
\end{equation}
Of course it may recast in the form of a theory with memory by 
integrating \eqref{intvode}.
We will see that the model \eqref{relpoly} has very interesting
structural properties.

\subsection{The augmented relaxation system}

The format of the stress relaxation model \eqref{relpoly}
is motivated (and was guided) by the enlargement structure
of the polyconvex elastodynamics system \eqref{pelas} described in section 
\ref{secvisco}.

Indeed, \eqref{relpoly} can be embedded into the augmented relaxation system
\begin{equation}
\label{polyrelax}
\begin{aligned}
\del_{t} v_{i} - \del_{\alpha} 
\Big ( \frac{\del \Phi^{A}}{\del F_{i \alpha}} T^{A} \Big ) &=0
\\
\del_{t} \Xi^{A} - \del_{\alpha} 
\Big ( \frac{\del \Phi^{A}}{\del F_{i \alpha}} v_{i} \Big ) &=0
\\
\del_{t} 
\Big ( T^{A} - \frac{\del \sigma_{I}}{\del \Xi^{A}}(\Xi) \Big )
&= -\frac{1}{\eps} 
\Big ( T^{A} - \frac{\del \sigma_{E}}{\del \Xi^{A}}(\Xi)  \Big )
\end{aligned}
\end{equation}
The stress function in the model \eqref{polyrelax} reads:
\begin{equation}
\label{stressv}
S_{i \alpha} = T^{A} \frac{\del \Phi^{A}}{\del F_{i \alpha}} \, .
\end{equation}
Note that as $\eps \to 0$ the stress $S_{i \alpha}$ formally approximates  the limiting stress
$$
S_{i \alpha , \infty} = T^{A} (\Xi) \Big |_{eq} \frac{\del \Phi^{A}}{\del F_{i \alpha}}
= \frac{\del \sigma_{E}}{\del \Xi^{A}}(\Xi)  
\frac{\del \Phi^{A}}{\del F_{i \alpha}}
$$
and thus \eqref{polyrelax} will approximate the extended elastodynamics system
\begin{equation}
\label{exelas}
\begin{aligned}
\del_{t} v_{i} - 
\del_{\alpha} \Big ( \frac{\del \sigma_{E}}{\del \Xi^{A}}(\Xi) 
\frac{\del \Phi^{A}}{\del F_{i \alpha}} \Big )  &=0
\\
\del_{t} \Xi^{A} - \del_{\alpha} 
\Big ( \frac{\del \Phi^{A}}{\del F_{i \alpha}} v_{i} \Big ) &=0
\end{aligned}
\end{equation}
Observe also that solutions of \eqref{relpoly} satisfy the kinematic 
constraints \eqref{kinnl} and thus, for a polyconvex stored energy,
the relaxation system \eqref{relpoly} enjoys the same relation with the system 
\eqref{polyrelax}  as the equations of polyconvex elastodynamics \eqref{elas} 
have with the system \eqref{enlarged}.

Next, we develop the Chapman-Enskog expansion for
the relaxation limit from 
\eqref{polyrelax} to \eqref{exelas}. 
Introduce the expansion for the internal variable $T^{A}$
$$
T^{A , \eps} = T^{A}_{0} + \eps T^{A}_{1} + O(\eps^{2})
$$
and, accordingly,
$$
S_{i \alpha}^{\eps} = T^{A}_{0}  \frac{\del \Phi^{A}}{\del F_{i \alpha}}
+ \eps T^{A}_{1}  \frac{\del \Phi^{A}}{\del F_{i \alpha}} + O(\eps^{2})
$$
to \eqref{polyrelax} in order to obtain
$$
\begin{aligned}
T^{A}_{0} & = \frac{\del \sigma_{E}}{\del \Xi^{A}}(\Xi)
\\
\del_{t} \Big ( \frac{\del \sigma_{E}}{\del \Xi^{A}}(\Xi) -
\frac{\del \sigma_{I}}{\del \Xi^{A}}(\Xi)   \Big )
&= - T^{A}_{1} + O(\eps)
\end{aligned}
$$
The effective momentum equation becomes
$$
\begin{aligned}
\del_{t} v_{i} - 
\pa \Big ( T_{0}^{A}  \frac{\del \Phi^{A}}{\del F_{i \alpha}} \Big )
&= \eps \pa 
\Big ( T_{1}^{A}  \frac{\del \Phi^{A}}{\del F_{i \alpha}} \Big )
+ O(\eps^{2})
\\
&= \eps \pa 
 ( D_{i \alpha}^{j \beta} \del_{\beta} v_{j}  )
+ O(\eps^{2})
\end{aligned}
$$
where
\begin{equation}
\label{defd}
D_{i \alpha}^{j \beta} :=
\frac{\del^{2}(\sigma_{I}-\sigma_{E})}{\del \Xi^{A} \del \Xi^{B}}
\frac{\del \Phi^{A}}{\del F_{i \alpha}} 
\frac{\del \Phi^{B}}{\del F_{j \beta}}
\end{equation}
In summary, the Chapman-Enskog expansion shows that as $\eps \to 0$ the relaxation process is approximated
by the hyperbolic-parabolic system
$$
\begin{aligned}
\del_{t} \Xi^{A} - \del_{\alpha} 
\Big ( \frac{\del \Phi^{A}}{\del F_{i \alpha}} v_{i} \Big ) &=0
\\
\del_{t} v_{i} - 
\pa \Big ( T_{0}^{A}  \frac{\del \Phi^{A}}{\del F_{i \alpha}} \Big )
&= \eps \pa   ( D_{i \alpha}^{j \beta} \del_{\beta} v_{j}  )
\end{aligned}
$$
Note that for $\Sigma : = \sigma_{I} - \sigma_{E}$ convex the 
diffusivity tensor $D$ satisfies 
the ellipticity condition
$
D_{i \alpha}^{j \beta} M_{i \alpha} M_{j \beta} \ge 0 \, ,
\;  \forall M \in \R^{3 \times 3} \, .
$
The latter is stronger than the Legendre-Hadamard condition, and 
is achieved when both the instantaneous potential 
$\sigma_{I} \circ \Phi$ and the equilibrium potential $\sigma_{E} \circ \Phi$
are polyconvex.

\subsection{Entropy of the augmented relaxation system}
\label{sec3a}

We next construct an entropy for the augmented relaxation system:
If  a function $\Psi (\Xi, \tau)$ can be constructed defined
$\forall \, (\Xi, \tau)$ and satisfying
\begin{equation}
\label{entrodef}
\begin{aligned}
\frac{\del \Psi}{\del \Xi^{A}}(\Xi,\tau) &= T^{A} = 
\frac{\del \sigma_{I}(\Xi)}{\del \Xi^{A}} + \tau^{A}  
\\
\frac{\del \Psi}{\del \tau^{A}} \, 
&\Big ( \tau^{A} - \frac{\del (\sigma_{E} - \sigma_{I} )}{\del \Xi^{A}}\Big )
\ge 0 \qquad  \forall \,  (\Xi, \tau) \, ,
\end{aligned}
\end{equation}
then the relaxation system is endowed with an H-theorem
\begin {equation}
\del_{t} \Big ( \frac{1}{2} |v|^{2} + \Psi (\Xi, \tau) \Big ) 
- \del_{\alpha} \big (  v_{i} S_{i \alpha} \big )  + \frac{1}{\eps}
\frac{\del \Psi}{\del \tau^{A}} \, 
\big ( \tau^{A} - \frac{\del (\sigma_{E} - \sigma_{I} )}{\del \Xi^{A}}\big )
= 0 \, .
\end{equation}

This entropy identity is based on the null-Lagrangian property \eqref{nullag2}
and follows, using \eqref{polyrelax}, \eqref{nullag2}
and \eqref{entrodef}, by the computation
\begin{align*}
\del_{t} \Big ( \frac{1}{2} |v|^{2} &+ \Psi (\Xi, \tau) \Big )
= v_{i} \del_{t} v_{i} + \frac{\del \Psi}{\del \Xi^{A}} \del_{t} \Xi^{A} 
   + \frac{\del \Psi}{\del \tau^{A}} \del_{t} \tau^{A}
\\
&= v_{i} \pa S_{i \alpha} + \frac{\del \Psi}{\del \Xi^{A}}
\pa ( \frac{\del \Phi^{A}}{\del F_{i \alpha}} v_{i}  )
 + \frac{\del \Psi}{\del \tau^{A}} \del_{t} \tau^{A}
 \\
&=  v_{i} \pa S_{i \alpha} + \frac{\del \Psi}{\del \Xi^{A}}
\frac{\del \Phi^{A}}{\del F_{i \alpha}} \pa v_{i}
 + \frac{\del \Psi}{\del \tau^{A}} \del_{t} \tau^{A}
 \\
&= \pa (v_{i} S_{i \alpha} ) - 
\frac{1}{\eps}
\frac{\del \Psi}{\del \tau^{A}} \, 
\big ( \tau^{A} - \frac{\del (\sigma_{E} - \sigma_{I} )}{\del \Xi^{A}}\big )
\end{align*}

Our next objective is to examine the solvability of \eqref{entrodef} and 
study the convexity of the entropy. Integrating \eqref{entrodef}$_{1}$, we see that
\begin{equation}
\label{potenergy}
\Psi(\Xi, \tau) = \sigma_{I} (\Xi) + \Xi \cdot \tau + G(\tau)
\end{equation}
where the integrating factor $G(\tau)$ has to be selected so that it 
satisfies the inequality
\begin{equation}
\label{ineq2}
(\Xi + \nabla_{\tau} G) \cdot (\tau + \nabla_{\Xi} \Sigma)
\ge 0
\quad \forall \, (\Xi, \tau)
\end{equation}
where $\Sigma = \sigma_{I} - \sigma_{E}$.
Regarding the solvability of \eqref{ineq2},  we show

\begin{lemma}
\label{lemex}
The functions  $G(\tau)$  and $\Sigma(\Xi)$ in $C^2 (\R^m)$ satisfy \eqref{ineq2} 
if and only if
\begin{equation}
\label{char}
\begin{cases}
\Xi + \nabla_{\tau} G = 0 \quad  \text{ iff} \quad
\tau + \nabla_{\Xi} \Sigma = 0   & \\
G \; \text{is convex}  & \\
\Sigma \; \text{is convex}  & \\
\end{cases}
\end{equation}
\end{lemma}

\noindent
Equation \eqref{char}$_1$ indicates that $G(\tau)$ and $\Sigma (\Xi)$ are connected through the Legendre transformation.

\begin{proof}
We first show that \eqref{ineq2} implies \eqref{char}. 
Fix $\Xi^{0}$, $\tau^{0}$
such that $\Xi^{0} + \nabla_{\tau} G (\tau^{0}) = 0$. Consider a fixed direction
$e^{A}$ and the increment along this direction $\Xi = \Xi^{0} + t e^{A}$. Then 
\eqref{ineq2} implies that 
$e^{A} \cdot (\tau^{0} + \nabla_{\Xi} \Sigma (\Xi^{0})) = 0$ for every
direction $e^{A}$ and thus 
$\tau^{0} + \nabla_{\Xi} \Sigma (\Xi^{0}) = 0$.
Similarly, if $\Xi^{0}$, $\tau^{0}$ are such that 
$\tau^{0} + \nabla_{\Xi} \Sigma (\Xi^{0}) = 0$ then also
$\Xi^{0} + \nabla_{\tau} G (\tau^{0}) = 0$. This proves the first statement
in \eqref{char}.

Fix now $\Xi_{1}$, $\Xi_{2}$ and let 
$\tau_{2} = - \nabla_{\Xi} \Sigma (\Xi_{2})$. 
Then $\Xi_{2} = -\nabla_{\tau} G (\tau_{2})$, \eqref{ineq2} is 
rewritten as
\begin{equation}
\label{ineqconvex}
(\Xi_{1} - \Xi_{2} ) \cdot 
\big ( \nabla_{\Xi} \Sigma (\Xi_{1}) - \nabla_{\Xi}  \Sigma 
(\Xi_{2}) \big ) 
\ge 0
\end{equation}
and $\Sigma$ is convex. A similar argument shows that $G$ is convex.

The converse is proved by re-expressing 
the convexity inequality \eqref{ineqconvex} in the form \eqref{ineq2} 
by using the first statement in the 
right of \eqref{char}.
\end{proof}

Lemma \ref{lemex} indicates that the solvability of \eqref{entrodef} 
is equivalent to the convexity of  $\Sigma := \sigma_{I} - \sigma_{E}$. 
To complete the details of the construction of $\Psi$, we assume for simplicity
that
\begin{equation}
\label{hypo0}
\nabla^{2}_{\Xi} \Sigma > 0 \quad \text{and} \quad
\nabla_{\Xi} \Sigma : \R^{D} \to \R^{D} \; \text{is onto},
\tag{h$_{0}$}
\end{equation}
with $D=19$ for $d=3$ and $D=5$ for $d=2$.
Define the inverse map $(\nabla_{\Xi} \Sigma)^{-1} : \R^{D} \to \R^{D}$, 
and let $h(\tau) = - (\nabla_{\Xi} \Sigma)^{-1} (-\tau)$. 
Then $\nabla_{\tau} h$ is symmetric and the 
differential system $\nabla_{\tau} G = h$ is solvable. Its solution 
$G$ is a convex function and satisfies
\begin{equation}
\begin{aligned}
\label{prop1}
\nabla_{\tau} G (\tau) &=  - \left ( \nabla_{\Xi} \Sigma \right )^{-1} (-\tau)
\\
\nabla^{2}_{\tau} G (\tau) &= 
\left [ \nabla^{2}_{\Xi} \Sigma  (- \nabla_{\tau} G) \right ]^{-1}
\end{aligned}
\end{equation}
$\Psi$ is defined by \eqref{potenergy} with $G$ as above.
Observe that, by \eqref{entrodef} and \eqref{char},
\begin{equation}
\label{prop2}
\begin{aligned}
\frac{\del \Psi}{\del \Xi^{A}} (\Xi, -\nabla_{\Xi} \Sigma ) &=
\frac{\del \sigma_{E}}{\del \Xi^{A}} (\Xi)
\\
\frac{\del \Psi}{\del \tau^{A}} (\Xi, -\nabla_{\Xi} \Sigma ) &=
\Xi^{A} + \frac{\del G}{\del \tau^{A}} \Big |_{\tau^{A} = - 
\frac{\del(\sigma_{I} - \sigma_{E})}{\del \Xi^{A}} } = 0
\end{aligned}
\end{equation}
and, by selecting a normalization constant,
\begin{equation}
\label{prop3}
\Psi (\Xi , - \nabla_{\Xi} \Sigma  ) = \sigma_{E} (\Xi)
\end{equation}

We next consider the convexity of $\Psi (\Xi, \tau)$ determined by the 
matrix 
$$
\nabla^{2}_{(\Xi, \tau)} \Psi =
    \begin{bmatrix} 
    \nabla^{2}_{\Xi} \sigma_{I}  &  \mathbb{I}
\\ 
\mathbb{I}   & \nabla^{2}_{\tau} G
\end{bmatrix}
$$

\begin{lemma}
\label{lemcon}
Let $\Sigma = \sigma_{I} - \sigma_{E}$ satisfy \eqref{hypo0} and assume
that $\sigma_{I}$, $\Sigma$ satisfy for $\gamma_{I} > \gamma_{v} > 0$ 
\begin{equation}
\label{hypo1}
\nabla^{2}_{\Xi} \sigma_{I} \ge \gamma_{I} \mathbb{I}_{m}  > \gamma_{v} \mathbb{I}_{m} \ge
\nabla^{2}_{\Xi} \Sigma > 0
\tag{h$_{1}$}
\end{equation}
Then for some $\delta >0$ we have
$$
\nabla^{2}_{(\Xi, \tau)} \Psi \ge \delta \;  \mathbb{I}_{(\Xi, \tau)}
$$
\end{lemma}

\begin{proof} Using \eqref{hypo1} and \eqref{prop1}$_2$ we estimate the Hessian of $\Psi$ as follows
$$
\begin{aligned}
\left ( \nabla^{2}_{(\Xi, \tau)} \Psi \right ) \; (\Xi, \tau) \cdot (\Xi, \tau)
&= 
(\nabla_{\Xi}^{2} \sigma_{I}) \Xi \cdot \Xi + 2 \Xi \cdot \tau +
\big ( \nabla^{2}_{\Xi} \Sigma \big )^{-1} \tau \cdot \tau
\\
&\ge \gamma_{I} |\Xi|^{2} + 2 \Xi \cdot \tau + \frac{1}{\gamma_{v}} |\tau|^{2}
\\
&\ge
( \gamma_{I} - \delta ) |\Xi|^{2} + 
  \Big ( \frac{1}{\gamma_{v}} - \frac{1}{\delta} \Big ) |\tau|^{2} \, .
\end{aligned}
$$
The coefficients  can be made positive definite by selecting 
$\gamma_{I} > \delta > \gamma_{v}$.
\end{proof}

\begin{remark} \rm
Hypothesis \eqref{hypo1} implies that  $\sigma_{E}$ must be convex,
which dictates that the limiting equations arise from a polyconvex 
energy.
\end{remark}

\subsection{Relative entropy for the augmented system}
\label{sec3b}

Next we compare a solution $(v, \Xi, \tau)$ of the system 
\eqref{polyrelax} with a solution $(\vh, \xh)$ of the extended 
elastodynamics system \eqref{exelas}, using a relative 
entropy calculation in the spirit of \cite{LT06, Tza04}.

The relative entropy is defined by taking the Taylor polynomial of a nonequilibrium
 relative to a Maxwellian solution
$$
\begin{aligned}
\mathcal E_{r} &:= \frac{1}{2} |v - \hat v|^{2} +
     \Psi (\Xi , \tau ) - 
     \Psi \Big ( \xh , \frac{\del (\sigma_{E} - \sigma_{I} )}{\del \Xi} (\xh) 
     \Big )
     \\
     &-
     \frac{\del \Psi}{\del \Xi } 
     (\xh, - \frac{\del \Sigma}{\del \Xi} (\xh) ) \cdot (\Xi  - \xh) -
     \frac{\del \Psi}{\del \tau }  \Big (
     \xh, - \frac{\del \Sigma}{\del \Xi} (\xh) \Big ) 
     \cdot \Big (\tau   -  \frac{\del (\sigma_{E} - \sigma_{I})}{\del \Xi} (\xh) 
     \Big )
     \\
     \end{aligned}
$$
where $\Sigma = \sigma_{I} - \sigma_{E}$.
By \eqref{prop2}, \eqref{prop3}, $\mathcal E_{r}$ has
the simple form
\begin{equation}
\label{relentropy}
\mathcal E_{r}
= \frac{1}{2} |v - \hat v|^{2} + \Psi (\Xi , T - \frac{\del \sigma_{I}}{\del \Xi}) - 
\sigma_{E} (\xh)  - \frac{\del \sigma_{E} }{\del \Xi} (\xh) \cdot 
(\Xi - \xh)
\end{equation}

We now recall the identities: The H-theorem for the relaxation 
approximation
\allowdisplaybreaks{   
\begin{align}
\label{hthm}
\del_{t} \Big ( \frac{1}{2} |v|^{2} &+ \Psi (\Xi, \tau) \Big )
- \pa (v_{i} S_{i \alpha} ) +
\frac{1}{\eps}
\frac{\del \Psi}{\del \tau^{A}} \, 
\big ( \tau^{A} - \frac{\del (\sigma_{E} - \sigma_{I})}{\del \Xi^{A}}\big )
= 0
\end{align}
and the energy equation for the extended elastodynamics system
\begin{align}
\label{lenergy}
\del_{t} \Big ( \frac{1}{2} |\vh|^{2} + \sigma_{E} (\xh)  \Big )
- \pa \Big (   \frac{\del \sigma_{E} }{\del \Xi^{A}} (\xh) 
\frac{\del \Phi^{A}}{\del F_{i \alpha} } (\fh) \vh_{i} \Big ) = 0
\end{align}

Finally we form the difference equations
\begin{equation*}
            \pt(v_{i}-\vh_{i}) - \pa \left ( T^{A}
       \frac{\partial \Phi^{A}}{\partial F_{i\alpha}}(F) 
       - \frac{\partial \sigma_{E} }{\partial \Xi^{A}} (\xh)
       \frac{\partial \Phi^{A}}{\partial F_{i\alpha}}(\fh)  \right ) = 
       0 ,
\end{equation*}
        \begin{equation*}
            \pt (\Xi^{A} - \xh^{A} ) - \pa \left (
            \frac{\partial \Phi^{A}}{\partial F_{i\alpha}}(F) \, v_{i} 
       - \frac{\partial \Phi^{A}}{\partial F_{i\alpha}}(\fh) 
       \, \vh_{i} \right ) =0
        \end{equation*}
and compute using \eqref{polyrelax} and \eqref{exelas} to obtain
\begin{align}
    &\pt \Big [ \vh_{i}(v_{i}-\vh_{i}) + \frac{\partial
      \sigma_{E} }{\partial \Xi^{A}} (\xh) \,  ( \Xi^{A} - \xh^{A} )\Big ]
 \nonumber\\
    & \qquad\ - \pa \left [ \vh_{i}\left ( T^{A} 
      \frac{\partial \Phi^{A}}{\partial F_{i\alpha}} (F) 
        - \frac{\partial  \sigma_{E} }{\partial \Xi^{A}} (\xh)
   \frac{\partial \Phi^{A}}{\partial F_{i\alpha}}(\fh)  \right )
       \right.
\nonumber\\
       &\qquad \qquad \qquad \ \left.
       + \frac{\partial  \sigma_{E} }{\partial \Xi^{A}} (\xh)
       \left (\frac{\partial \Phi^{A}}{\partial
       F_{i\alpha}}(F) \, v_{i} - \frac{\partial \Phi^{A}}{\partial
       F_{i\alpha}}(\fh) \, \vh_{i} \right )  \right ]
            \nonumber\\
             &= (\pt \vh_{i}) (v_{i}-\vh_{i}) + \pt \left (
       \frac{\partial  \sigma_{E} }{\partial \Xi^{A}} (\xh)
            \right ) (\Xi^{A} - \xh^{A} )
\nonumber\\
             & \qquad\ - \pa \vh_{i} \left (
                T^{A} \frac{\partial \Phi^{A}}{\partial F_{i\alpha}} (F)
       - \frac{\partial  \sigma_{E} }{\partial \Xi^{A}} (\xh)
       \frac{\partial \Phi^{A}}{\partial  F_{i\alpha}}(\fh)  \right ) 
\nonumber\\
      &\qquad\qquad \qquad 
      - \pa \left ( \frac{\partial  \sigma_{E} }{\partial \Xi^{A}} (\xh)
      \right )\left ( \frac{\partial \Phi^{A}}{\partial
       F_{i\alpha}}(F) \, v_{i} - \frac{\partial \Phi^{A}}{\partial
       F_{i\alpha}}(\fh) \, \vh_{i}\right ) 
\nonumber \\
&= - \pa \left ( \frac{\partial  \sigma_{E} }{\partial \Xi^{A}} (\xh) \right )
 \Big ( \frac{\partial \Phi^{A}}{\partial F_{i\alpha}}(F) 
       - \frac{\partial \Phi^{A}}{\partial F_{i\alpha}}(\fh) 
       \Big ) \, v_{i} 
\nonumber \\
&\quad - \pa \vh_{i} 
\Big [ T^{A} \frac{\partial \Phi^{A}}{\partial F_{i\alpha}} (F) 
 - \frac{\partial  \sigma_{E} }{\partial \Xi^{A}} (\xh)
       \frac{\partial \Phi^{A}}{\partial  F_{i\alpha}}(\fh)
 - \frac{\del^{2} \sigma_{E}(\xh)}{\del \Xi^{A} \del \Xi^{B}} 
    \frac{\partial \Phi^{A}}{\partial  F_{i\alpha}}(\fh) (\Xi^{B} - 
    \xh^{B} ) \Big ]
\nonumber  \\
&= : I
\label{corr}
\end{align}

By rearranging the terms and using the null-Lagrangian property \eqref{nullag2}
we may rewrite $I$ in the form
\begin{align}
I &= 
- \pa \Big [  
\vh_{i} \frac{\partial  \sigma_{E} }{\partial \Xi^{A}} (\xh) 
   \left ( \frac{\partial \Phi^{A}}{\partial F_{i\alpha}}(F) 
    - \frac{\partial \Phi^{A}}{\partial F_{i\alpha}}(\fh) \right )
\Big ]
\nonumber \\
&\quad - \pa \left ( \frac{\partial  \sigma_{E} }{\partial \Xi^{A}} (\xh)
            \right )
  \left ( \frac{\partial \Phi^{A}}{\partial F_{i\alpha}}(F) 
    - \frac{\partial \Phi^{A}}{\partial F_{i\alpha}}(\fh) \right )
    ( v_{i} - \vh_{i} )
\nonumber \\
&\quad - (\pa \vh_{i}) \frac{\partial \Phi^{A}}{\partial  F_{i\alpha}}(\fh)
\Big ( \frac{\partial  \sigma_{E} }{\partial \Xi^{A}} (\Xi)
 - \frac{\partial  \sigma_{E} }{\partial \Xi^{A}} (\xh)
 - \frac{\del^{2} \sigma_{E}}{\del \Xi^{A} \del \Xi^{B}} (\xh) 
    (\Xi^{B} - \xh^{B} ) \Big )
\nonumber \\
&\quad - (\pa \vh_{i}) \left ( 
  \frac{\partial  \sigma_{E} }{\partial \Xi^{A}} (\Xi)
 - \frac{\partial  \sigma_{E} }{\partial \Xi^{A}} (\xh)  
      \right ) 
  \left ( \frac{\partial \Phi^{A}}{\partial F_{i\alpha}}(F) 
    - \frac{\partial \Phi^{A}}{\partial F_{i\alpha}}(\fh) \right )
\nonumber \\
&\quad  - (\pa \vh_{i}) 
\left ( T^{A} - \frac{\partial  \sigma_{E} }{\partial \Xi^{A}} (\Xi)
\right )
\frac{\partial \Phi^{A}}{\partial F_{i\alpha}}(F)
\nonumber \\
&= - \pa \left [  
\vh_{i} \frac{\partial  \sigma_{E} }{\partial \Xi^{A}} (\xh) 
   \left ( \frac{\partial \Phi^{A}}{\partial F_{i\alpha}}(F) 
    - \frac{\partial \Phi^{A}}{\partial F_{i\alpha}}(\fh) \right )
\right ]
- Q_{1} -Q_{2} - Q_{3} - L
\label{error}
\end{align}
That is the term $I$ is written as the sum of a divergence term 
plus the quadratic terms $Q_{i}$ plus a linear term $L$ that is controlled 
by the distance from equilibrium.

Combining \eqref{hthm}, \eqref{lenergy}, \eqref{corr} and 
\eqref{error} we arrive at the relative entropy identity
\begin{equation}
\label{relenpoly}
\begin{aligned}
\del_{t} \mathcal{E}_{r} &- \pa \mathcal{F}_{\alpha, r} +
\frac{1}{\eps} D
= Q_{1}  + Q_{2} + Q_{3} + L
\end{aligned}
\end{equation}
where the flux is 
\begin{equation}
\mathcal{F}_{\alpha, r} := 
\left ( T^{A} - \frac{\partial  \sigma_{E} }{\partial \Xi^{A}} (\xh)
\right ) (v_{i} - \vh_{i}) \frac{\partial \Phi^{A}}{\partial F_{i\alpha}}(F)
\end{equation}
the dissipation is
\begin{equation}
\frac{1}{\eps} D = 
\frac{1}{\eps}
\frac{\del \Psi}{\del \tau^{A}}
\Big  ( \Xi,  T - \frac{\del \sigma_{I} }{\del \Xi} \Big ) \,
\big ( T^{A} - \frac{\del \sigma_{E} }{\del \Xi^{A}}\big )
\end{equation}
the quadratic errors $Q_{i}$ are 
\begin{equation}
\label{quad}
\begin{aligned}
Q_{1} &= \pa \left ( \frac{\partial  \sigma_{E} }{\partial \Xi^{A}} (\xh)
            \right )
  \left ( \frac{\partial \Phi^{A}}{\partial F_{i\alpha}}(F) 
    - \frac{\partial \Phi^{A}}{\partial F_{i\alpha}}(\fh) \right )
    ( v_{i} - \vh_{i} )
 \\
 Q_{2} &= (\pa \vh_{i}) \frac{\partial \Phi^{A}}{\partial  F_{i\alpha}}(\fh)
\Big ( \frac{\partial  \sigma_{E} }{\partial \Xi^{A}} (\Xi)
 - \frac{\partial  \sigma_{E} }{\partial \Xi^{A}} (\xh)
 - \frac{\del^{2} \sigma_{E}(\xh) }{\del \Xi^{A} \del \Xi^{B}}
    (\Xi^{B} - \xh^{B} ) \Big )
 \\
 Q_{3} &=  (\pa \vh_{i}) \left ( 
  \frac{\partial  \sigma_{E} }{\partial \Xi^{A}} (\Xi)
 - \frac{\partial  \sigma_{E} }{\partial \Xi^{A}} (\xh)  
      \right ) 
  \left ( \frac{\partial \Phi^{A}}{\partial F_{i\alpha}}(F) 
    - \frac{\partial \Phi^{A}}{\partial F_{i\alpha}}(\fh) \right )
\end{aligned}
\end{equation}
and the linear error $L$ is
\begin{equation}
\label{line}
L = (\pa \vh_{i}) 
\left ( T^{A} - \frac{\partial  \sigma_{E} }{\partial \Xi^{A}} (\Xi)
\right )
\frac{\partial \Phi^{A}}{\partial F_{i\alpha}}(F)
\end{equation}
Identity \eqref{relenpoly} is the key on which the stability and 
convergence analysis of section \ref{secstab} is based.
}      


\section{Stability theorem}
\label{secstab}

Consider a family of smooth solutions $\{(v^{\eps}, F^{\eps}, \tau^{\eps})\}_{\eps > 0}$ , $\tau^\eps = T^\eps - \nabla_\Xi \sigma_I (\Phi (F^\eps))$,
to the relaxation system \eqref{relpoly}.
We wish to compare them with a smooth solution $(\vh, \fh)$ of the equations 
of polyconvex elastodynamics \eqref{pelas}.
For simplicity of notation, we drop the $\eps$-dependence from the solution of the relaxation system.
The data $F_{0}$ and $\fh_{0}$ are taken gradients; 
this property is propagated by \eqref{pelas}$_2$ and both $F$ and $\fh$ are gradients for all times.
The function $(v, \Phi(F), \tau)$ is a smooth solution of the augmented
relaxation system \eqref{polyrelax}
while the function $(\vh, \Phi (\fh))$ satisfies
the extended elastodynamics equations \eqref{exelas}. From the results of section 
\ref{sec3b},
smooth solutions of \eqref{polyrelax} and \eqref{exelas}
satisfy \eqref{relenpoly}.  

The identity \eqref{relenpoly} is inherited 
by $(v, \Phi(F), \tau)$ and $(\vh, \Phi (\fh))$.
The resulting relative energy and associated flux,
\begin{equation}
\label{relaxrelen}
\begin{aligned}
e_{r} &= \mathcal{E}_{r} \Big (v, \Phi(F), \tau \; \big | \;  \vh, \Phi(\fh), 
\frac{\del (\sigma_{E} - \sigma_{I})}{\del \Xi} (\Phi(\fh))  \Big )
\\
  &=
\frac{1}{2} |v-\vh|^{2} 
 + \Psi \Big ( \Phi(F), T - \frac{\del \sigma_{I}}{\del \Xi}(\Phi(F)) \Big )
- \sigma_{E}(\Phi(\fh))
\\
&\qquad \qquad
    -\frac{\partial \sigma_{E}}{\partial \Xi^{A}}(\Phi(\fh))
               (\Phi(F)^{A} - \Phi(\fh)^{A})\, ,
\end{aligned}
\end{equation}
\begin{equation}
\label{relaxrelflux}
\begin{aligned}
f_{\alpha} &=
\mathcal{F}_{\alpha , r} \Big (v, \Phi(F), \tau \; \big | \;  \vh, \Phi(\fh), 
\frac{\del (\sigma_{E} - \sigma_{I})}{\del \Xi} (\Phi(\fh))  \Big )
\\
  &=
   \left ( T^{A} -
    \frac{\partial \sigma_{E}}{\partial \Xi^{A}}(\Phi(\fh))
    \right)
(v_{i}-\vh_{i})
\frac{\partial \Phi^{A}}{\partial F_{i\alpha}}(F) \, ,
\end{aligned}
\end{equation}
satisfy
\begin{equation}
\label{finrelen}
\begin{aligned}
\del_{t} e_{r} &- \pa f_{\alpha} +
\frac{1}{\eps} D
= Q_{1}  + Q_{2} + Q_{3} + L
\end{aligned}
\end{equation}
where the $Q_{i}$, $L$ and $D$ are now computed for $\Xi = \Phi(F)$ 
and $\xh = \Phi (\fh)$.

We prove convergence of the relaxation system 
to polyconvex elastodynamics so long as the limit solution is smooth.

\begin{theorem}
\label{theorelax}
Let $(v^{\e}, F^{\e}, T^{\e})$, $F^{\e} = \nabla y^{\e}$, be
smooth solutions of \eqref{relpoly} and $(\vh, \fh)$, 
$\fh = \nabla \yh$, 
a smooth solution of \eqref{pelas}, defined on $\R^{d} \times [0,T]$
and decaying fast as $|x|\to\infty$.
The relative energy $e_{r}$  defined in \eqref{relaxrelen}
satisfies \eqref{finrelen}.
Assume that $\sigma_{I}$, $\sigma_{E}$ satisfy for some constants
$\gamma_{I} > \gamma_{v} > 0$ and $M > 0$ the
hypotheses 
     \begin{align}
     \label{hypa1}
     \nabla^{2} \sigma_{I}  \ge \gamma_{I} \mathbb{I} &> \gamma_{v} \mathbb{I} \ge
     \nabla^{2}  ( \sigma_{I} - \sigma_{E}) > 0 \, ,
     \tag {h$_{1}$}
     \\
     \label{hypo2}
     |\nabla^{2}  \sigma_{E}|  &\le M  \, , \quad
     |\nabla^{3} \sigma_{E} | \le M \, .
     \tag {h$_{2}$}
     \end{align}
There exists a constant
$s$ and
$C =C (T, \gamma_{I},\gamma_{v}, M, \nabla \vh, \nabla \fh)>0$ independent of
$\e$ such that
       \begin{equation*}
          \int_{\R^{d}}  e_{r} (x,t) dx \le C
          \left ( \int_{\R^{d}}  e_{r} (x,0) dx  +\e \right)\, .
       \end{equation*}
In particular, if the data satisfy
\begin{equation*}
\int_{\R^{d}}  e^{\eps}_{r}
(x,0) dx \longrightarrow 0 \, , \quad \text{as $\e\downarrow 0$}\, ,
\end{equation*}
then
       \begin{align*}
  \sup_{t\in [0,T]} \int_{\R^{d}} |v^{\e} - \vh |^{2}
  + |F^{\e} - \fh|^{2} + |\tau^{\e} - \tau_{\infty} (\fh)|^{2} dx 
  \longrightarrow 0 \, ,
       \end{align*}
where $\tau_{\infty} (\fh) = \frac{ \del (\sigma_{E} - 
\sigma_{I})}{\del \Xi} (\Phi (\fh))$.
\end{theorem}

\begin{proof}
The equation \eqref{finrelen}, 
$$
\begin{aligned}
\del_{t} e_{r} &+  \del_{\alpha}  f_{\alpha}
+ \frac{1}{\eps} D = J \, ,
\end{aligned}
$$
is integrated on $\R^{d} \times (0,t)$ and gives
\begin{align}
\int_{\R^{d}} e_{r} (x,t) dx &- \int_{\R^{d}} e_{r} (x,0) dx 
\nonumber 
\\
&\qquad + \frac{1}{\eps}  \int_{0}^{t} \int_{\R^{d}} D dx d\tau
= \int_{0}^{t} \int_{\R^{d}} J dx d\tau
\label{es1}
\end{align}

From lemma \ref{lemcon} and \eqref{relentropy} we see that there
exists a positive constant $c = c(\gamma_{I}, \gamma_{v})$ such that
$$
\mathcal {E}_{r} \ge c \left ( |v - \vh|^{2} + |\Xi - \xh||^{2} +
|\tau - \frac{\del (\sigma_{E} - \sigma_{I})}{\del \Xi} (\xh) |^{2} 
\right )
$$
and thus, by \eqref{relaxrelen},
$$
e_{r} \ge c \big ( |v - \vh|^{2} + |\Phi (F) - \Phi (\fh)|^{2} 
    + | \tau - \tau_{\infty} (\fh)|^{2} \big ) \, .
$$

Note that
\begin{equation}
\label{pr1}
\begin{aligned}
D &:=
\frac{\del \Psi}{\del \tau_{A}} \Big ( \tau^{A} - \frac{\del 
(\sigma_{E} - \sigma_{I})}{\del \Xi^{A}} \Big )
\\
&= ( \Xi + \nabla_{\tau} G) \cdot (\tau + \nabla_{\Xi} \Sigma )
\\
&= (\nabla_{\tau} G (\tau) - \nabla_{\tau} G ( -\nabla_{\Xi} \Sigma ) ) 
  \cdot (\tau + \nabla_{\Xi} \Sigma)
\\
&\ge
\left ( \min \nabla_{\tau}^{2} G \right ) |\tau + \nabla_{\Xi} 
\Sigma|^{2}
\\
&\ge
\frac{1}{\gamma_{v}} | \tau  - \nabla_{\Xi} (\sigma_{E} - \sigma_{I})|^{2}
\end{aligned}
\end{equation}

Let now $C$ be a positive constant depending on the $L^{\infty}$-norm 
of $\vh$, $\fh$, $\del_{\alpha} \vh$, $\del_{\alpha} \fh$ and the 
constants $\gamma_{I}$, $\gamma_{v}$ and $M$.  On account of \eqref{hypo2} and the smoothness of $(\vh, \fh)$, the term
$Q_2$ is of quadratic growth on $|\Xi -\xh| = |\Phi(F) - \Phi(\fh)|$. Using  \eqref{quad},
\eqref{hypo2}, and \eqref{line} we have
\begin{align*}
\int_{\R^{d}}  |Q_{1}| dx &\leq C \int_{\R^{d}}
|v-\vh|^{2}  + \Big |\frac{\partial \Phi}{\partial F}(F) -
\frac{\partial \Phi}{\partial F} (\fh) \Big |^{2}  dx\, ,
\\
\int_{\R^{d}}  |Q_{2}| dx &\leq C \int_{\R^{d}} |\Phi(F) - \Phi 
(\fh)|^{2} dx \, ,
\\
\int_{\R^{d}}  |Q_{3}| dx 
   &\leq C \int_{\R^{d}}
|\Phi(F) - \Phi  (\fh)|^{2}  
+ \Big |\frac{\partial \Phi}{\partial F}(F) -
\frac{\partial \Phi}{\partial F} (\fh) \Big |^{2} dx\, ,
\end{align*}
and
\begin{align*}
\int_{\R^{d}} |L| dx \le \frac{1}{\eps} \frac{1}{2 \gamma_{v}} 
\int_{\R^{d}} | \tau  - \nabla_{\Xi} (\sigma_{E} - \sigma_{I})|^{2} dx
+ C \eps \int_{\R^{d}} | \frac{\del \Phi}{\del F}(F) |^{2} dx  \, .
\end{align*}
 From the identities
$$
\frac{\del \det F}{\del F_{i \alpha}} = (\cof F)_{i \alpha} \, ,
\quad
\frac{\del (\cof F)_{i\alpha}}{\del F_{j \beta}} = \e_{ijk}\e_{\alpha
\beta \gamma} F_{k \gamma}\, ,
$$
we have
$$
\Big | \frac{\partial \Phi}{\partial F}(F) -
\frac{\partial \Phi}{\partial F}(\fh) \Big | \le C
| \Phi (F) - \Phi (\fh) |\, .
$$

Combining with \eqref{pr1} and \eqref{es1} we obtain
\begin{align}
&\int_{\R^{d}} e_{r} (x,t) dx 
+ 
\frac{1}{2 \eps \gamma_{v}} 
\int_{\R^{d}} | \tau  - \nabla_{\Xi} (\sigma_{E} - \sigma_{I})|^{2} dx
\nonumber
\\
&\qquad = \int_{\R^{d}} e_{r} (x,0) dx 
+  C \int_{0}^{t} \int_{\R^{d}} e_{r} (x, \tau) dx d\tau
\nonumber
\\
&\qquad \qquad
+ \eps C \int_{0}^{t} \int_{\R^{d}} | \frac{\del \Phi}{\del F}(F) |^{2} 
dx d\tau
\label{es2}
\end{align}

 The H-estimate implies that solution of the relaxation system 
\eqref{relpoly} satisfy the uniform (in $\eps$) bounds
\begin{align}
&\int_{\R^{d}} |v|^{2} + |\Phi (F)|^{2} + |\tau|^{2} dx 
+ 
\frac{1}{\eps \gamma_{v}} 
\int_{\R^{d}} | \tau  - \nabla_{\Xi} (\sigma_{E} - \sigma_{I})|^{2} dx
\nonumber
\\
&\qquad \le C \int_{\R^{d}} |v_{0}|^{2} + \Psi( \Phi(F_{0}) , \tau_{0}) dx \le 
O(1)
\end{align}
The result then follows from \eqref{es2} via Gronwall's inequality.
\end{proof}

\section{Gas dynamics in Eulerian coordinates}
\label{secgd}

As an example we work out the relaxation model that results when applying \eqref{relpoly} to  the equations of isentropic gas
dynamics. In preparation, we review the classical transformation of a balance law from Lagrangean to Eulerian
coordinates ({\it e.g.} \cite[Sec 2.2]{Daf10}).

\subsection{Transformation from Lagrangean to Eulerian coordinates}

Consider a motion $y (\cdot , t)  : \cR  \to \cR_t$  that maps a reference configuration $\cR$ onto the current configuration
$\cR_t$, for each $t  \in [0,T]$. The Lagrangean coordinates in the reference configuration are denoted by $x = (x_\alpha)_{\alpha = 1, ... , d}$
and the Eulerian coordinates in the current configuration by $y = ( y_j)_{j = 1, ... , d}$ with $d$ the (common) dimension of the ambient spaces.
The map $y(\cdot, t)$ is assumed globally invertible and a bi-Lipschitz homeomosphism, and we denote by
$$
v_i  = \frac{\del y_i }{\del t} \, , \quad F_{i \alpha}  = \frac{ \del y_i}{\del x_\alpha}
$$
the velocity and deformation gradient respectively.

Suppose  the fields $\phi = \phi (x,t)$, $\psi_\alpha = \psi_\alpha (x,t)$ and $p = p(x,t)$ are defined in the Lagrangian frame and 
satisfy the balance law
\begin{equation}
\label{lagbl}
\del_t \phi (x,t)  = \del_\alpha \psi_\alpha (x,t)  + p(x,t) \, .
\end{equation}
The fields $\phi$, $\psi_\alpha$ and $p$ can be scalar or vector fields.
The Lagrangian balance law \eqref{lagbl} can be transformed to an equivalent balance law expressed on the Eulerian coordinate frame
\begin{equation}
\label{eulpbl}
\del_t \left (  \frac{\phi}{\det F} \circ y^{-1} \right )  + \del_{y_j} \left (   \big ( \frac{\phi}{\det F} v_j  \big )  \circ y^{-1} \right )  = 
\del_{y_j} \left (  \big ( \frac{\psi_\alpha F_{j \alpha} }{\det F}  \big ) \circ y^{-1} \right )   +  \frac{p}{\det F} \circ y^{-1} \, .
\end{equation}
In expressing \eqref{eulpbl} we have used $y^{-1} (y,t)$  to be the inverse (in $x$) map of $y(x,t)$. This dependence is
often implied when stating the balance law in Eulerian coordinates and it is commonplace to write \eqref{eulpbl},
 using a somewhat ambivalent notation, in the form
\begin{equation}
\label{eulbl}
\del_t \left (  \frac{\phi}{\det F}  \right )  + \del_{y_j} \left (    \frac{\phi}{\det F}   u_j  \right )  = 
\del_{y_j} \left (   \frac{\psi_\alpha F_{j \alpha} }{\det F}   \right )   +  \frac{p}{\det F}  \, ,
\end{equation}
where $u_j = v_j \circ y^{-1}$  (or equivalently $u_j ( y(x,t), t) = v_j (x,t)$) stands for the velocity expressed
in Eulerian coordinates.

\subsection{Expression of gas dynamics in Lagrangean coordinates}
Consider  now the system of isentropic gas dynamics in Eulerian coordinates
\begin{align}
\label{gdmass}
\del_t  \rho + \del_j (\rho u_j) &= 0
\\
\label{gdmom}
\del_t ( \rho u_i ) + \del_j (\rho u_i u_j ) + \del_i p(\rho) &= 0
\end{align}
where $\rho = \rho(y,t)$ the density, $u = u (y,t)$ the velocity, $y \in \R^3$, and the pressure $p(\rho) > 0$ 
satisfies $p' (\rho) > 0$ which  guarantees hyperbolicity. The system of isentropic gas dynamics satisfies the energy conservation 
equation
\begin{equation}
\label{gdener}
\del_t (  \frac{1}{2} \rho |u|^2 + \rho e(\rho) ) + \del_j \big ( u_j (\frac{1}{2} \rho |u|^2 + \rho e(\rho) ) \big ) + \del_j (p(\rho) u_j ) = 0
\end{equation}
where the internal energy function $e(\rho)$ is related to the pressure through the usual relation
\begin{equation}
\label{inenpr}
e'(\rho) = \frac{p(\rho)}{\rho^2} > 0 \, .
\end{equation}
Note that $(\rho e)'' = \frac{p'}{\rho} > 0$ and that $\eta(\rho , m) = \frac{1}{2} \frac{|m|^2}{\rho}  + \rho e(\rho)$ 
is convex in the variables $(\rho, m)$,  $m = \rho u$ the momentum.

We proceed to calculate the associated Lagrangian form of the system \eqref{gdmass}-\eqref{gdmom}. For the velocity field $u(y,t)$ 
the initial value problem
\begin{equation}
\label{defmot}
\begin{cases}
\del_t y_i = u_i (y,t) & \\
y(x,0) = x  & x \in \cR \\
\end{cases}
\end{equation}
determines the motion $y(x,t)$. The local solvability of \eqref{defmot} is guaranteed for sufficiently smooth vector fields $u$
but the solution is not necessarily globally well defined. Here we will not discuss these important aspects and will proceed formally.
Given $y(x,t)$ we define $F = \nabla y$, $v = \del_t y$ and recall Abel's formula
$$
\del_t \det F = \div_y u \det F \, .
$$
Using the correspondence between the Lagrangean \eqref{lagbl} and Eulerian \eqref{eulbl} form of the balance law, we transform
the balance of mass equation \eqref{gdmass} to the form
$$
\del_t  (\rho \det F) = 0
$$
This implies that $\rho \det F =: \rho_0 (x)$ is independent of time. By assigning the reference density of the current configuration
(here selected as the $t=0$ instance of the current configuration) to be $\rho_0 (x) = 1$, we obtain 
\begin{equation}
\label{rdet}
\rho = \frac{1}{\det F} \, .
\end{equation}
In turn, using the  relations
$$
(F^{-1})_{\alpha i} F_{j \alpha}  = \delta_{i j} \, , \quad  (F^{-1})_{\alpha i} = \frac{1}{\det F}  ( \cof F )_{\alpha i} \, ,
$$
and \eqref{rdet}, \eqref{gdmom} is transformed into the Lagrangian form
\begin{equation}
\del_t  v_i  = \del_\alpha \big ( - p(\rho) (\det F) (F^{-1})_{\alpha i} ) =  \del_\alpha \big ( - p(\rho)  (\cof F)_{i \alpha} )
\end{equation}
Note the correspondence with the standard definition of the Cauchy stress for gas dynamics
$T_{i j} = - p(\rho) \delta_{ij}$ and its association with the Piola-Kirchhoff stress
$$
S_{i \alpha}  = T_{i j}( \det F ) ( F^{-1})_{\alpha j} = - p(\rho) (\det F) (F^{-1})_{\alpha i}
$$
Similarly, the energy equation \eqref{gdener} transforms to the Lagrangean form
$$
\del_t \big (  \frac{1}{2} |v|^2 + e(\rho)   \big ) = \del_\alpha \big ( - p(\rho) v_i (\cof F)_{i \alpha} \big )
$$
To the above equations we may add the nonlinear transport relation
$$
\del_t \det F = \del_\alpha ( v_i (\cof F)_{i \alpha} )
$$
which is a consequence of the null Lagrangians \eqref{nullag2} and part of \eqref{kinnl}.

In summary, the full set of Lagrangean equations for gas dynamics is
\begin{align}
\label{lagcom}
\del_t F_{i \alpha} &= \del_\alpha v_i
\\
\label{lagiden}
\del_t \det F &= \del_\alpha ( v_i (\cof F)_{i \alpha} )
\\
\label{lagmom}
\del_t v_i  &= \del_\alpha \big ( - p \Big ( \frac{1}{\det F} \Big ) (\cof F)_{i \alpha} \big )
\end{align}
and the Lagrangean form of the energy is
\begin{equation}
\label{lagener}
\del_t \Big (  \frac{1}{2} |v|^2 + e \Big ( \frac{1}{\det F} \Big )  \Big ) = \del_\alpha \Big (   - p \Big ( \frac{1}{\det F} \Big ) v_i  (\cof F)_{i \alpha}  \Big ) \, .
\end{equation}
The stored energy $W$ is of the form
\begin{equation}
\label{storen}
W(F) = e \Big ( \frac{1}{\det F} \Big ) =  g (\det F)  
\end{equation}
where
$$
\begin{aligned}
g(w) &: =  e \left ( \frac{1}{w} \right )
\\
\frac{dg}{dw} &= e' \Big ( \frac{1}{w} \Big ) \Big ( - \frac{1}{w^2} \Big ) = - p \Big ( \frac{1}{w} \Big ) \, , 
\\
\frac{d^2g}{dw^2}  &= p'  \Big ( \frac{1}{w} \Big ) \frac{1}{w^2} > 0 \, , 
\end{aligned}
$$
Hence $W$ is polyconvex,   the system \eqref{lagcom}-\eqref{lagmom} fits into the framework of polyconvex elasticity with the identification 
$g(w) : =  e \left ( \frac{1}{w} \right )$, and of course it is associated with an extended symmetrizable system  of the form
\eqref{enlarged} for the variables $(F, \Xi)$ with $\Xi = (F, w)$ in the present case.

\subsection{A relaxation model for gas dynamics in Lagrangean coordinates}
We consider now the relaxation model 
\begin{equation}
\label{gdrel}
\begin{aligned}
\del_t v_i  &=  \del_\alpha \Big ( \big [ - p_I  \big (  \frac{1}{\det F} \big ) + \tau \big ]  \cof F_{i \alpha} \Big )
\\
\del_t F_{i \alpha} &= \del_\alpha v_i
\\
\del_t \tau &= -\frac{1}{\eps} \Big ( \tau + p_E  \big (  \frac{1}{\det F} \big ) - p_I  \big (  \frac{1}{\det F} \big ) \Big )
\end{aligned}
\end{equation}
This model is a special case of the model \eqref{relpoly} with a scalar internal variable 
$$
T = - p_I  \big (  \frac{1}{\det F} \big ) + \tau
$$
We assume that the instantaneous $p_I(\rho)$ and equilibrium $p_E(\rho)$ pressure functions are strictly positive
and satisfy
\begin{equation}
\label{ass0}
\begin{aligned}
p'_I (\rho ) &> 0 \, , \quad e'_I (\rho) = \frac{p_I (\rho)}{\rho^2}
\\
p'_E (\rho ) &> 0 \, , \quad e'_E (\rho) = \frac{p_E (\rho)}{\rho^2}
\end{aligned}
\tag{a$_{0}$}
\end{equation}
with $e_I (\rho)$ and $e_E (\rho)$ the associated instantaneous and equilibrium internal energy functions.

Furthermore, \eqref{gdrel} is associated with the augmented relaxation system ({\it cf} \eqref{polyrelax}) 
\begin{equation}
\label{gdaugrel}
\begin{aligned}
\del_t v_i  &=  \del_\alpha \Big ( \big [ - p_I  \big (  \frac{1}{w} \big ) + \tau \big ]  \cof F_{i \alpha} \Big )
\\
\del_t F_{i \alpha} &= \del_\alpha v_i
\\
\del_t w  &=  \del_\alpha \big (  (\cof F)_{i \alpha} v_i  \big )
\\
\del_t \tau &= -\frac{1}{\eps} \Big ( \tau + p_E  \big (  \frac{1}{w} \big ) - p_I  \big (  \frac{1}{w}  \big ) \Big )
\end{aligned}
\end{equation}
with $w > 0$, and the theory developed in section \ref{secrel} can be applied directly to \eqref{gdaugrel} with the following
identifications
$$
\begin{aligned}
\sigma_I (w) &= e_I (\frac{1}{w} ) \, , \quad \frac{d \sigma_I}{dw} = - p_I ( \frac{1}{w})
\\
\sigma_E (w) &= e_E (\frac{1}{w} ) \, , \quad \frac{d \sigma_E}{dw} = - p_E ( \frac{1}{w})
\end{aligned}
$$
where by \eqref{ass0} both $\sigma_I (w)$ and $\sigma_w(w)$ are convex.

Following the procedure of section \ref{sec3a}, we obtain an entropy for the augmented relaxation system and
in turn for the reduced system \eqref{gdrel}. By multiplying \eqref{gdaugrel}$_1$ by $v_i$, \eqref{gdaugrel}$_3$ by
$\big ( - p_I  \big (  \frac{1}{w} \big ) + \tau \big )$, and \eqref{gdaugrel}$_4$ by $(w + G' (\tau))$ we obtain 
the entropy equation
\begin{equation}
\label{gdrelent}
\begin{aligned}
\del_t \Big (  \frac{1}{2} |v|^2 + e_I \big ( \frac{1}{w} \big ) + w \tau + G(\tau)  \Big )  
&=
\del_\alpha \Big (  \big [  - p_I   \big ( \frac{1}{w} \big ) + \tau \big ]  (\cof F)_{i \alpha} v_i \Big ) 
\\
&\quad    -   \frac{1}{\eps} \big ( w + G'(\tau) \big ) \Big (  \tau  +  (p_E - p_I) \big ( \frac{1}{w} \big ) \Big )
\end{aligned}
\end{equation}
where
$$
\begin{aligned}
\sigma_I (w) &:=  e_I \big ( \frac{1}{w} \big ) = - \int_1^w p_I \big ( \frac{1}{s} \big ) ds
\\
G(\tau) &:= - \int_1^\tau  \frac{1}{ (p_I - p_E)^{-1} (s) } ds 
\end{aligned}
$$
Indeed, if the pressure functions satisfy the hypothesis
\begin{equation}
\label{hypoth1}
(p_I - p_E)' (\rho) > 0  \quad \forall \rho > 0
\tag{a$_{1}$}
\end{equation}
then $(p_I - p_E)^{-1}$ and $G(\tau)$ are well defined, $\Sigma (w) = (\sigma_I - \sigma_E)(w)$ is convex and Lemma \ref{lemex}
guarantees the existence of a global, dissipative entropy 
$$
\Psi (w,\tau) = \sigma_I (w) + w \tau + G(\tau) \, .
$$
Using Lemma \ref{lemcon} it follows that the entropy $\Psi (w,\tau)$ is convex in $(w,\tau)$ provided
\begin{equation}
\frac{p'_I ( \frac{1}{w})}{w^2} \ge \frac{  (p_I - p_E)' ( \frac{1}{\bar w} )}{ {\bar w}^2} \, , \quad \forall w, \bar w > 0 \, .
\tag{a$_{2}$}
\end{equation}

\subsection{Expression of the relaxation model in Eulerian coordinates}
We next apply again the transformation procedure from Lagrangean to Eulerian coordinates
to express the model \eqref{gdrel} in Eulerian coordinates. We recall the expression  $\rho = \frac{1}{\det F}$ and
note that \eqref{gdrel} when expressed in Eulerian coordinates gives
\begin{equation}
\label{gdpressrelax}
\begin{aligned}
\del_t  \rho + \del_j (\rho u_j) &= 0
\\
\del_t ( \rho u_i ) + \del_j (\rho u_i u_j ) &=  \del_j  \big (  (- p_I(\rho) + \tau) \delta_{ij} \big )
\\
\del_t (\rho \tau) + \del_j ( \rho u_j \tau) &=  - \frac{1}{\eps} \rho \big ( \tau - p_I (\rho) + p_E (\rho) \big )
\end{aligned}
\end{equation}
This is a pressure relaxation model with two pressures an instantaneous and an equilibrium pressure.
Models of that general type have previously been observed in the literature, see for example \cite{CG67, Sul98}.
Such models correspond to a mechanism of relaxation of pressures with an
instantaneous and an equilibrium pressure response, the latter associated with the long time response of the model
in the way outlined  in section \ref{secrel}, and  are endowed with and entropy function defined locally 
(near equilibrium) which is dissipative \cite{Sul98}. The present model is endowed with a {\it globally defined}
entropy function. This can be seen by reverting the entropy dissipation identity \eqref{gdrelent} into Eulerian
coordinates. The process gives 
\begin{equation}
\label{gdeulrelent}
\begin{aligned}
&\del_t \left [ \frac{1}{2} \rho |v|^2 + \rho \big ( e_I (\rho) + \frac{1}{\rho} \tau + G(\tau) \big ) \right ]
+ \del_j \left ( u_j \big [ \frac{1}{2} \rho |v|^2 + \rho \big ( e_I (\rho) + \frac{1}{\rho} \tau + G(\tau) \big ) \big ] \right )
\\
&\qquad\qquad  = \del_j  \left ( (- p_I(\rho) + \tau) u_j \right )
-\frac{1}{\eps} \rho  \big ( \tau - (p_I - p_E) (\rho) \big ) \left ( \frac{1}{\rho} - \frac{1}{ (p_I - p_E)^{-1} (\tau)} \right )
\end{aligned}
\end{equation}
The existence of globally defined entropy relaxation functions is noted in \cite{CP97} in the context 
of internal energy relaxation models for gas dynamics and for general models with internal variables in
\cite{Tza99}. The present model provides another example that enjoys
this feature and is associated with pressure relaxation and is related to the
polyconvexity property of the elasticity system.

We finish by checking the conditions under which the above expressions are well defined. It is instructive
to check that directly. We always operate under the framework of \eqref{ass0} and let $P(\rho) = ( p_I - p_E )(\rho)$.
The entropy will be dissipative provided
\begin{equation}
\label{dissstr}
\rho (\tau - P(\rho) )  \Big ( \frac{1}{\tau} - \frac{1}{ P^{-1} (\rho) } \Big )  \ge  0 \, , \quad  \forall \rho, \tau > 0 \, .
\end{equation}
The equation \eqref{dissstr} holds if and only if the function $P$ is strictly increasing, that is if \eqref{hypoth1} holds.

Next, we examine the convexity of the function
\begin{equation}
\rho E(\rho, \tau) := \rho \left ( e_I (\rho) + \frac{1}{\rho} \tau + G(\tau) \right )
\end{equation}
by checking the eigenvalues of the Hessian matrix
\begin{equation*}
    \begin{pmatrix}
       \displaystyle{ \frac{d^2}{d\rho^2} (\rho e_I) } &
       \displaystyle{  G'(\tau) }
       \\
	\displaystyle{  G'(\tau) } & 
	\displaystyle{ \rho G'' (\tau) }
    \end{pmatrix}  \, .
\end{equation*}
The eigenvalues are strictly positive if
$$
\begin{aligned}
(\rho e_I )'' = \frac{p'_I}{\rho} > 0  \, ,
\\
(\rho e_I )'' \rho G''(\tau) - (G'(\tau) )^2 > 0 \, .
\end{aligned}
$$
To express the second condition, note that if $\tau = P(\bar \rho)$ then $\bar \rho = P^{-1} (\tau)$ and 
$$
G' (\tau) = - \frac{1}{ P^{-1} (\tau)} = - \frac{1}{\bar \rho} \, , \quad
G''(\tau) = \frac{1}{ [ P^{-1} (\tau)]^2} \frac{1}{ P' ( P^{-1} (\tau) )} = \frac{1}{ {\bar \rho}^2 P'(\bar \rho) }
$$
In view of \eqref{ass0}, the convexity of $\rho E(\rho, \tau)$ is equivalent to the condition
\begin{equation}
\label{hyppressdiff}
p'_I (\rho) > ( p'_I - p'_E) (\bar \rho) > 0  \, ,  \qquad \forall \rho, \bar \rho > 0 \, .
\tag {a$_3$}
\end{equation}
This can be combined with the fact
that the function $\frac{|m|^2}{\rho}$ is convex in $(\rho, m)$ to conclude that
the under \eqref{hyppressdiff} the entropy
\begin{equation}
H (\rho, \tau, m) := \frac{1}{2} \frac{|m|^2}{\rho} + \rho \big ( e_I (\rho) + \frac{1}{\rho} \tau + G(\tau) \big )
\end{equation}
is convex in $(\rho, \tau, m)$ with $m = \rho u$ the momentum.

\subsection*{Acknowledgements}
Partially supported by the EU FP7-REGPOT project "Archimedes Center for Modeling, 
Analysis and Computation" and  by the "Aristeia" program of the Greek Secretariat for Research.


\begin{thebibliography}{10}

\bibitem{Bal77}
J.M. Ball,
Convexity conditions and existence theorems in nonlinear elasticity,
{\em Arch. Rational Mech. Anal.} \textbf{63} (1977), 337-403.

\bibitem{Bal02}
J.M. Ball, Some open problems in elasticity
In: \emph{Geometry, mechanics and dynamics}, 
Springer, New York, 2002, pp. 3-59.

\bibitem{Bou05}
F. Bouchut, 
{\it Nonlinear stability of finite volume methods for hyperbolic conservation laws and well-balanced 
schemes for sources}. Frontiers in Mathematics. BirkhŠuser Verlag, Basel, 2004.

\bibitem{Bre04}
Y.~Brenier,  Hydrodynamic structure of the augmented Born-Infeld
equations, \emph{Arch. {R}ational {M}ech. {A}nal.} \textbf{172} (2004), 65-91.

\bibitem{BKW07}
F. Bouchut, C. Klingenberg, and K. Waagan,
A multiwave approximate Riemann solver for ideal MHD based on relaxation. I. Theoretical framework. 
{\em Numer. Math.} {\bf 108} (2007), 7Ð42.

\bibitem{Cia93}
P.G. Ciarlet,
{\em Mathematical Elasticity}, Vol. 1,
North Holland, (1993).

\bibitem{CLL94}
G.-Q. Chen, C.D. Levermore and T.-P. Liu, 
Hyperbolic conservation laws with stiff relaxation terms and entropy, 
   \emph{Comm. {P}ure and {A}ppl. {M}ath.}
  \textbf{47} (1994), 787--830.
  
\bibitem{CN59}
B. Coleman and W. Noll,
On the Thermostatics of Continuous Media.
{\em Arch. Rat. Mech. Anal.}, {\bf 4} (1959), 97-128.

\bibitem{CG67} B.D. Coleman and M.E. Gurtin,
Thermodynamics with internal state variables,
{\em J. Chem. Physics} {\bf 47} (1967), 597-613.

\bibitem{CP97} F. Coquel and B. Perthame,
Relaxation of energy and approximate Riemann solvers for
general pressure laws in fluid dynamics,
{\em SIAM J. Num. Anal.} {\bf 35} (1998), 2223-2249.

\bibitem{Daf79} 
C.M. Dafermos,
The second law of thermodynamics and stability,
\emph{Arch. Rational Mech. Analysis} \textbf{70} (1979), 167-179.

\bibitem{Daf86}
C.M. Dafermos, 
Quasilinear hyperbolic systems with involutions,
{\em Arch. Rational  Mech. Anal.} {\bf 94} (1986), 373-389.

\bibitem{Daf10}
C.M. Dafermos, \emph{Hyperbolic conservation laws in continuum physics}, 3rd edition,
Grundlehren der {M}athematischen {W}issenschaften, vol. 325,   Springer-Verlag,
Berlin, 2010.


\bibitem{DST01}
S.~Demoulini, D.M.A. Stuart and A.E. Tzavaras, A variational
   approximation scheme for three--dimensional elastodynamics with polyconvex
   energy, \emph{Arch. {R}ation. {M}ech. {A}nal.} \textbf{157} (2001), 325--344.

\bibitem{DST12}
S.~Demoulini, D.M.A. Stuart and A.E. Tzavaras, 
Weak-strong uniqueness of dissipative measure-valued solutions for polyconvex elastodynamics, 
\emph{Arch. {R}ation. {M}ech. {A}nal.} \textbf{205} (2012), 927-961.
   
\bibitem{DiP79} R.J. DiPerna,
Uniqueness of solutions to hyperbolic conservation laws,
\emph{Indiana Univ. Math. J.}  \textbf{28} (1979), 137-188.

\bibitem{FM87} 
C. Faciu and M. Mihailescu-Suliciu,
The energy in one-dimensional rate-type semilinear
viscoelasticity,
  \emph{Int. J. Solids Structures} {\bf 23} (1987), 1505-1520.

\bibitem{LT06} 
C. Lattanzio and A. Tzavaras,
Structural properties of stress relaxation and convergence
from viscoelasticity to polyconvex elastodynamics.
\emph{Arch. Rational Mech. Anal.}  {\bf 180} (2006), 449-492.

\bibitem{Lax71} P.D. Lax,
Shock waves and entropy,
in: ''Contributions to Nonlinear Functional Analysis."
E.H.~Zarantonello, ed. New York: Academic Press, 1971, pp. 603-634.


\bibitem{Qin98}
T. Qin, Symmetrizing nonlinear elastodynamic system,
\emph{J. Elasticity} \textbf{50} (1998), 245-252.

\bibitem{Ser04}
D.~Serre, Hyperbolicity of the {N}onlinear {M}odels of {M}axwell's
{E}quations,  \emph{Arch. {R}ational {M}ech. {A}nal.} \textbf{172} (2004),
    309--331.
    
\bibitem{Sle13}
M.~Slemrod, {\it Lectures on the Isometric Embedding Problem $(M^n, g) \to {\mathbb R}^m$,
$m = \frac{n}{2} (n+1)$}, Unpublished Lecture Notes, 2013.

\bibitem{Sul98}
I.~Suliciu, On the thermodynamics of rate-type fluids and phase
   transitions. {I}. {R}ate-type fluids, 
   \emph{Internat. J. Engrg. Sci.} \textbf{36}
   (1998), no.~9, 921--947.

\bibitem{Tza99}
A.E.~Tzavaras, Materials with internal variables and relaxation to
   conservation {L}aws, \emph{Arch. {R}ational {M}ech. {A}nal.} \textbf{146} (1999),
   129--155.

\bibitem{Tza04}
A.E.~Tzavaras, Relative entropy in hyperbolic relaxation \emph{Comm. 
Math. Sci} \textbf{3} (2005), 119-132.

\bibitem{Tza05} 
A.~Tzavaras, A relaxation theory with polyconvex entropy function converging to elastodynamics.
(2005), (unpublished).

\end{thebibliography}

\end{document}